\documentclass[12pt,reqno]{amsart}
\usepackage{amssymb}
\usepackage{eurosym}
\usepackage{epsfig}
\usepackage{amsmath}
\usepackage{amsfonts}
\usepackage{pgf,tikz}
\usepackage{enumerate}
\usepackage{cite}

\usepackage{graphicx}
\usepackage{caption}
\usepackage{subcaption}

\newtheorem{theorem}{\sc Theorem}[section]
\newtheorem{proposition}{\sc Proposition}[section]

\newtheorem{definition}{\sc Definition}[section]
\newtheorem{remark}{\sc Remark}[section]

\newtheorem{example}{\sc Example}[section]

\def\qed{\hbox to 0pt{}\hfill$\rlap{$\sqcap$}\sqcup$\medbreak}

\title[An existence result in annular regions times conical shells]{An existence result in annular regions times conical shells and its application to nonlinear Poisson systems}

\begin{document}
	\author[G. Infante]{Gennaro Infante}
	\address{Gennaro Infante, Dipartimento di Matematica e Informatica, Universit\`{a} della
		Calabria, 87036 Arcavacata di Rende, Cosenza, Italy}%
	\email{gennaro.infante@unical.it}%
	
	\author[G. Mascali]{Giovanni Mascali}
	\address{Giovanni Mascali, Dipartimento di Matematica e Informatica, Universit\`{a} della
		Calabria, 87036 Arcavacata di Rende, Cosenza, Italy}%
	\email{giovanni.mascali@unical.it}%
	
	\author[J. Rodr\'iguez--L\'opez]{Jorge Rodr\'iguez--L\'opez}
	\address{Jorge Rodr\'iguez--L\'opez, CITMAga \& Departamento de Estat\'{\i}stica, An\'alise Matem\'atica e Optimizaci\'on, Universidade de Santiago de Compostela,  15782, Facultade de Matem\'aticas, Campus Vida, Santiago, Spain}
	\email{jorgerodriguez.lopez@usc.es}%

\date{}

\begin{abstract}
We provide a new existence result for abstract nonlinear operator systems in normed spaces, by means of topological methods. The solution is located within the product of annular regions and conical shells. The theoretical result possesses a wide range of applicability, which, for concreteness, we illustrate in the context of systems of nonlinear Poisson equations subject to homogeneous Dirichlet boundary conditions. For the latter problem we obtain existence and localization of solutions having all components nontrivial. This is also illustrated with an explicit example in which we also furnish a numerically approximated solution, consistent with the theoretical results. We conclude with an application of our results to a reaction--diffusion Lotka--Volterra system with source terms for competing species.
\end{abstract}


\maketitle
\section{Introduction}
It is well known that the solvability of elliptic systems plays a key role when modelling real world phenomena~\cite{Pao}. Under the point of view of applications it is of interest to obtain the explicit solution (if possible) or at least  
as much qualitative information about the solution as possible, which can be useful also for devising suitable numerical schemes.
Various methods can be used to provide existence and localization of solutions. Insofar as topological methods are concerned, a classical approach is to rewrite the differential problem as an operator system and the localization of the solution of the latter system yields qualitative informations on the solution of the differential problem. To fix ideas, let us consider the following
 system of nonlinear Poisson equations subject to homogeneous Dirichlet boundary conditions
\begin{equation}\label{Poisson}
	\left\{\begin{array}{ll} -\Delta u=f(x,u,v), & \ \text{ in } \Omega, \\ -\Delta v=g(x,u,v), & \ \text{ in } \Omega, \\ u=v=0, & \ \text{ on } \partial\,\Omega,  \end{array} \right.
\end{equation}
where $\Omega\subset\mathbb{R}^n$ denotes the unit open ball and $f$ and $g$ are continuous functions. 
This kind of systems has been widely investigated by means of different methodologies, for example variational or topological, see for example the papers~\cite{alves-defig,  chzh07, chzh13, Cid-Infante, cui3, defigetal, gimmrp, lan1, lan2, lan3, ma2, zhch11, zou}, the  reviews~\cite{amann, defig, bruf}, and references therein. 

When using topological methods to solve the system~\eqref{Poisson}, if both the nonlinearities occurring in~\eqref{Poisson} are sign-changing, it is natural to seek solutions located in the product of two balls in suitable Banach spaces, while when the nonlinearities $f$ and $g$ are nonnegative, a natural choice is to look for solutions within suitable cones of positive functions, see, for instance, \cite{alves-defig, amann, chzh13, lan1, ma2}. 
An interesting case occurs when one of the nonlinearities is nonnegative and the other is sign changing. Our work aims to provide new results that fit precisely within this framework.
For this purpose, we study the following general abstract problem regarding  
the solvability of the operator system
\begin{equation}\label{system}
	\left\{\begin{array}{l} u_1=T_1(u_1,u_2), \\ u_2=T_2(u_1,u_2). \end{array} \right.
\end{equation}

The main abstract results are based on new fixed point index computations, which provide sufficient conditions for the existence of \textit{coexistence} fixed points $(u_1,u_2)$ for the operator $T=(T_1,T_2)$. The term \textit{coexistence}, already employed by Lan in \cite{lan3}, means that both components of the fixed point, $u_1$ and $u_2$, are non-trivial. If only one of the components is non-trivial, we say that the solution is  \textit{semi-trivial}, see for example~\cite{Dancer, LanLin}. 
In particular, to localize the solutions of the system \eqref{system}, we use compression-expansion homotopy type conditions in each component of the operator $T$, as in the \textit{vector version} of Krasnosel'ski\u{\i} fixed point theorem due to Precup \cite{PrecupFPT,PrecupSDC}. Compared with the original theorem by Precup, our main result (see Theorem \ref{th_starKras} below) applies for more general domains of the operator $T$ and, moreover, we obtain that its fixed point index is either $1$ or $-1$, which in particular ensures the existence of at least one fixed point. 

Similar computations of the fixed point index have been already deduced in \cite{LFP_JRL,InMaRo,JRL}, also in relation with Krasnosel'ski\u{\i}-Precup fixed point theorem. Unlike 
\cite{LFP_JRL,InMaRo,JRL}, here we work in the context of \textit{wedges}, instead of restricting the definition of the operators to the cartesian product of \textit{cones}, or of cones by closed convex subsets. We underline that this theoretical generalization has direct consequences in applications, providing new ways to localize the solutions to \eqref{system}. As far as we are aware, this is the first time that the solution to \eqref{system} is located in the cartesian product of an annular region times a conical shell. On the other hand, the manner how we compute the fixed point index here, which is based on its multiplicativity property, differs from that in \cite{LFP_JRL,JRL} and seems to be more intuitive. In addition, our results complement previous ones in the literature concerning systems of nonlinear equations, such as those in \cite{BGK23, ima}.

Going back to the applicability of the theoretical results, we consider the existence of solutions to a Dirichlet system of the form~\eqref{Poisson}.
We highlight that the nonlinearity $g$ may be sign-changing. Even so, we provide sufficient conditions for the existence of a solution~$(u,v)$ which is not semi-trivial, positive in $u$ and located within the product of a conical shell and an annulus. This localization provides interesting qualitative informations and quantitative estimates on the components of the solution.
We also provide a result useful to construct a  numerical approximation of the solution of~\eqref{Poisson}. We illustrate in an explicit example 
the constants that occur in our theory and we also exhibit numerical solutions that are consistent with our theoretical approach. Finally, as an application of the developed  theory, we consider a reaction--diffusion Lotka--Volterra system with source terms, which is widely used to study spatial ecological dynamics \cite{QiZh, Mur}, but can also be adapted to describe chemical reactions, or tumour growth \cite{ChDa,GaGa}.
We show that our results can also be effectively used to adjust the parameters in order to obtain solutions within a desired range.

\section{Fixed point index computations}
 
For the sake of completeness, we recall first some properties of the fixed point index for compact maps. Further details can be found in \cite{amann} or \cite[Chapter 12]{GraDug}.

We will say that a closed convex subset $K$ of a normed linear space $X$ is a \textit{wedge} if $\lambda\,u\in K$ for every $u\in K$ and for all $\lambda\geq 0$. A wedge $K$ is said to be a \textit{cone} if $K\cap (-K)=\{0\}$.

\begin{proposition}\label{prop_index}
	Let $C$ be a wedge of a normed space, $U\subset C$ be a bounded relatively open set and $S:\overline{U}\rightarrow C$ be a compact map such that $S$ has no fixed points on the boundary of $U$ (denoted by $\partial\, U$). Then the fixed point index  of $S$ on the set $U$ with respect to $C$, $i_{C}(S,U)$, has the following properties:
	\begin{enumerate}
		\item (Additivity) Let $U$ be the disjoint union of two open sets $U_1$ and $U_2$. If $0\not\in(I-S)(\overline{U}\setminus(U_1\cup U_2))$, then \[i_{C}(S,U)=i_{C}(S,U_1)+i_{C}(S,U_2).\]
		\item (Existence) If $i_{C}(S,U)\neq 0$, then there exists $u\in U$ such that $u=Su$.
		\item (Homotopy invariance) If $H:\overline{U}\times[0,1]\rightarrow C$ is a compact homotopy and $0\not\in(I-H)(\partial\,U\times[0,1])$, then
		\[i_{C}(H(\cdot,0),U)=i_{C}(H(\cdot,1),U).\]
		\item (Normalization) If $S$ is a constant map with $S(u)=u_0$ for every $u\in\overline{U}$, then
		\[i_{C}(S,U)=\left\{\begin{array}{ll} 1, & \text{ if } u_{0}\in U, \\ 0, & \text{ if } u_{0}\not\in\overline{U}. \end{array} \right. \]
		\item (Multiplicativity) For $j=1,2$, let $C_j$ be a wedge, $U_j\subset C_j$ be a open bounded set and $S_j:U_j\rightarrow C_j$ be a compact map fixed point free on the boundary of $U_j$. Then
		\[i_{C_1\times C_2}(S_1\times S_2,U_1\times U_2)=i_{C_1}(S_1,U_1)\cdot i_{C_2}(S_2,U_2). \]
	\end{enumerate}
\end{proposition}

\begin{proposition}\label{prop_ind01}
	Assume that $S$ satisfies the hypotheses of Proposition \ref{prop_index}. Let $U$ be a bounded relatively open subset of $C$ such that $0\in U$.
	\begin{enumerate}[$(a)$]
		\item If $\lambda u\neq Su$ for all $u\in \partial\, U$ and all $\lambda\geq 1$, then $i_{C}(S,U)=1$.
		\item If there exists $w\in C$ with $\left\|w\right\|\neq 0$ such that $u\neq Su+\lambda w$ for every $\lambda\geq 0$ and all $u\in \partial\, U$, then $i_{C}(S,U)=0$.
	\end{enumerate}	
\end{proposition}

\subsection{Star-shaped convex sets}

In the remaining part of this Section, let $(X,\left\|\cdot\right\|_X)$ and $(Y,\left\|\cdot\right\|_Y)$ be normed linear spaces and $K_1\subset X$, $K_2\subset Y$ two wedges. For simplicity, both norms $\left\|\cdot\right\|_X$ and $\left\|\cdot\right\|_Y$ will be denoted by $\left\|\cdot\right\|$.

Let us introduce the concept of \textit{star-convex set}. We refer the reader to the papers \cite{LR,LPR} for further properties of star-convex sets and a motivation for working with them in the context of Krasnosel'ski\u{\i} type compression--expansion fixed point theorems. 

\begin{definition}
	We say that a set $E\subset X$ is a star convex set if 
	\[\lambda\,x\in E \quad \text{for all } \lambda\in[0,1] \text{ and all } x\in E. \]
\end{definition}

Note that every convex set containing the zero is a star convex set. The reverse is not true.

\begin{example}
	In the normed space of continuous real functions defined in the compact interval $[0,1]$, $X=\mathcal{C}([0,1])$, the set 
	\[E=\left\{u\in X: u\geq 0, \ \min_{t\in[a,b]}u(t)<r \right\}, \]	
	with $[a,b]\subset [0,1]$ and $r>0$, is a star convex set. 
	
	However, $E$ is not convex: take $[a,b]=[0,1]$, $r=1/4$ and the functions $u_1(t)=t$, $u_2(t)=1-t$, $t\in[0,1]$, to check that $u_1,u_2\in E$ whereas $(u_1+u_2)/2\notin E$.
\end{example}

For each $i=1,2$, let $U_i$ and $V_i$ be bounded and relatively open subsets of $K_i$ such that
\begin{enumerate}
	\item $0\in V_i\subset \overline{V}_i\subset U_i$;
	\item $\overline{U}_i\setminus V_i$ is a retract of $\overline{U}_i$;
	\item $U_i$ and $V_i$ are star-convex sets.
\end{enumerate}

Now, we compute the fixed point index of a compact map defined in the Cartesian product of the sets $\overline{U}_i\setminus V_i$, $i=1,2$, under component-wise compression-expansion type assumptions.

	\begin{theorem}\label{th_starKras}	
	Assume that $T=(T_1,T_2):\left(\overline{U}_1\setminus V_1 \right)\times \left(\overline{U}_2\setminus V_2\right) \rightarrow K_1\times K_2$ is a compact map and, for each $i\in\{1,2\}$, there exists $h_i\in K_i\setminus\{0\}$ such that either of the following conditions are fulfilled in $\left(\overline{U}_1\setminus V_1 \right)\times \left(\overline{U}_2\setminus V_2\right)$:
	\begin{enumerate}
		\item[$(a)$] $T_i(u)+\mu\,h_i\neq u_i$ if $u_i\in\partial\,V_i$ and $\mu\geq 0$, and $T_i(u)\neq \lambda\, u_i$ if $u_i\in\partial\,U_i$ and $\lambda\geq 1$; or
		\item[$(b)$] $T_i(u)\neq \lambda\, u_i$ if $u_i\in\partial\,V_i$ and $\lambda\geq 1$, and $T_i(u)+\mu\,h_i\neq u_i$ if $u_i\in\partial\,U_i$ and $\mu\geq 0$. 
	\end{enumerate}

	Then the fixed point index of $T$ in $K_1\times K_2$ over $\left(U_1\setminus \overline{V}_1\right)\times \left(U_2\setminus \overline{V}_2\right)$ is well-defined and satisfies that
	\[i_{K_1\times K_2}\left(T,\left(U_1\setminus \overline{V}_1\right)\times \left(U_2\setminus \overline{V}_2\right)\right)=(-1)^k, \]
	where $k\in\{0,1,2\}$ is the number of times that condition $(b)$ is satisfied, $i=1,2$. 
	
	In particular, $T$ has at least one fixed point in $\left(U_1\setminus \overline{V}_1\right)\times \left(U_2\setminus \overline{V}_2\right)$.
\end{theorem}

\begin{proof}
	First of all, consider the retraction $\rho:\overline{U}_{1}\times \overline{U}_{2}\rightarrow \left(\overline{U}_1\setminus V_1 \right)\times \left(\overline{U}_2\setminus V_2\right)$ defined as $\rho(u_1,u_2):=(\rho_1(u_1),\rho_2(u_2))$, where $\rho_1$ is a retraction of $\overline{U}_{1}$ onto $\overline{U}_1\setminus V_1$ and $\rho_2$, a retraction of $\overline{U}_{2}$ onto $\overline{U}_2\setminus V_2$. 
	
	Next, define the following continuous extension of $T$ to the set $\overline{U}_{1}\times \overline{U}_{2}$, 
	\[N=(N_1,N_2):\overline{U}_{1}\times \overline{U}_{2}\rightarrow K_1\times K_2, \qquad N:=T\circ\rho.\] 
	The operator $N$ is compact and, moreover, the fact that $N=T\circ\rho$ together with assumptions $(a)$ and $(b)$ imply that for each $i\in\{1,2\}$ one of the following conditions is fulfilled in $\overline{U}_{1}\times \overline{U}_{2}$:
	\begin{enumerate}
		\item[$(a^*)$] $N_i(u)+\mu\,h_i\neq u_i$ if $u_i\in\partial\,V_i$ and $\mu\geq 0$, and $N_i(u)\neq \lambda\, u_i$ if $u_i\in\partial\,U_i$ and $\lambda\geq 1$; or
		\item[$(b^*)$] $N_i(u)\neq \lambda\, u_i$ if $u_i\in\partial\,V_i$ and $\lambda\geq 1$, and $N_i(u)+\mu\,h_i\neq u_i$ if $u_i\in\partial\,U_i$ and $\mu\geq 0$.
	\end{enumerate}
	
	Let us denote $C:=K_1\times K_2$. Now, for each $i\in\{1,2\}$, take $\mathcal{O}_i\in \{U_i,V_i \}$ and consider the homotopy $H:\overline{\mathcal{O}}_{1}\times \overline{\mathcal{O}}_{2}\times[0,1]\rightarrow K_1\times K_2$ given by
	\[H((u_1,u_2),t)=\left(N_1(u_1,t\,u_2), N_2(t\,u_1,u_2) \right). \]
	Clearly, $H$ is well-defined since $U_1$, $V_1$, $U_2$ and $V_2$ are star convex sets.
	Moreover, it follows from assumptions $(a^*)$ and $(b^*)$ (with $\mu=0$ and $\lambda=1$, respectively) that the homotopy is \textit{admissible} (i.e., $u\neq H(u,t)$ for all $u\in \partial\left(\overline{\mathcal{O}}_{1}\times \overline{\mathcal{O}}_{2} \right)$ and all $t\in[0,1]$) and thus the homotopy invariance of the fixed point index ensures that
	\[i_{C}(N,\mathcal{O}_{1}\times \mathcal{O}_{2})=i_{C}(H(\cdot,1),\mathcal{O}_{1}\times \mathcal{O}_{2})=i_{C}(H(\cdot,0),\mathcal{O}_{1}\times \mathcal{O}_{2}). \]
	Hence we have
	\[i_{C}(N,\mathcal{O}_{1}\times \mathcal{O}_{2})=i_{C}(\tilde{N},\mathcal{O}_{1}\times \mathcal{O}_{2}), \]
	where $\tilde{N}(u_1,u_2)=(\tilde{N}_1(u_1),\tilde{N}_2(u_2)):=(N_1(u_1,0),N_2(0,u_2))$. Therefore, the multiplicativity property of the fixed point index (see \cite[Chapter 12]{GraDug}) guarantees that 
	\begin{equation}\label{eq_multipli}
		i_{C}(N,\mathcal{O}_{1}\times \mathcal{O}_{2})=i_{K_1}(\tilde{N}_1,\mathcal{O}_{1})\cdot i_{K_2}(\tilde{N}_2,\mathcal{O}_{2}).
	\end{equation}

	Now, by the additivity property of the fixed point index we obtain that
	\[i_{C}(N,U_{1}\times U_{2})=i_{C}\left(N,\left(U_1\setminus \overline{V}_1\right)\times \left(U_2\setminus \overline{V}_2\right)\right)+i_{C}(N,V_{1}\times U_{2})+i_{C}\left(N,\left(U_1\setminus \overline{V}_1\right)\times V_2\right) \]
	and
	\[i_{C}(N,U_{1}\times V_{2})=i_{C}\left(N,\left(U_1\setminus \overline{V}_1\right)\times V_2\right)+i_{C}(N,V_{1}\times V_{2}). \]
	Therefore, we deduce that
	\begin{align}\notag
		i_{C}\left(N,\left(U_1\setminus \overline{V}_1\right)\times \left(U_2\setminus \overline{V}_2\right)\right)&= i_{C}(N,U_{1}\times U_{2})-i_{C}(N,U_{1}\times V_{2})-i_{C}(N,V_{1}\times U_{2}) \\ &\quad+i_{C}(N,V_{1}\times V_{2}).\label{eq_addit}
	\end{align}

	Let us consider four cases:
	
	\begin{enumerate}
	\item[Case 1:] \textit{$T_1$ and $T_2$ satisfy condition $(a)$}. 
	Then $N_1$ and $N_2$ satisfy condition $(a^*)$, so Proposition \ref{prop_ind01} ensures that 
	\[i_{K_i}(\tilde{N}_i,U_{i})=1 \quad \text{and} \quad i_{K_i}(\tilde{N}_i,V_{i})=0 \quad (i=1,2). \]
	Hence, by \eqref{eq_multipli}, we obtain the following computations of the fixed point index
	\begin{equation}\label{index0}
				i_{C}(N,U_{1}\times V_{2})=i_{C}(N,V_{1}\times U_{2})=i_{C}(N,V_{1}\times V_{2})=0
	\end{equation}
		and 
	\begin{equation}\label{index1}
				i_{C}(N,U_{1}\times U_{2})=1.
	\end{equation}
	By \eqref{eq_addit}, \eqref{index0} and \eqref{index1},
	\[i_{C}\left(N,\left(U_1\setminus \overline{V}_1\right)\times \left(U_2\setminus \overline{V}_2\right)\right)=1. \]
	Finally, since $N=T$ on the set $\left(\overline{U}_1\setminus V_1 \right)\times \left(\overline{U}_2\setminus V_2\right)$, we deduce 
	\[i_{C}\left(T,\left(U_1\setminus \overline{V}_1\right)\times \left(U_2\setminus \overline{V}_2\right)\right)=i_{C}\left(N,\left(U_1\setminus \overline{V}_1\right)\times \left(U_2\setminus \overline{V}_2\right)\right)=1.\]
	
	\item[Case 2:] \textit{$T_1$ satisfies condition $(a)$ and $T_2$, hypothesis $(b)$}. Then we have
	\[i_{K_1}(\tilde{N}_1,U_{1})=1=i_{K_2}(\tilde{N}_2,V_{2}) \quad \text{and} \quad i_{K_1}(\tilde{N}_1,V_{1})=0=i_{K_2}(\tilde{N}_2,U_{2}). \]
	By \eqref{eq_multipli}, we deduce that
	\[i_{C}(N,U_{1}\times U_{2})=i_{C}(N,V_{1}\times U_{2})=i_{C}(N,V_{1}\times V_{2})=0, \quad i_{C}(N,U_{1}\times V_{2})=1 \]
	and thus it follows from \eqref{eq_addit} that \[i_{C}\left(T,\left(U_1\setminus \overline{V}_1\right)\times \left(U_2\setminus \overline{V}_2\right)\right)=i_{C}\left(N,\left(U_1\setminus \overline{V}_1\right)\times \left(U_2\setminus \overline{V}_2\right)\right)=-1.\]
	
	\item[Case 3:] \textit{$T_1$ satisfies condition $(b)$ and $T_2$, hypothesis $(a)$}. It follows in an analogous way to Case 2 that $i_{C}\left(T,\left(U_1\setminus \overline{V}_1\right)\times \left(U_2\setminus \overline{V}_2\right)\right)=-1$.
	
	\item[Case 4:] \textit{Condition $(b)$ holds for both $T_1$ and $T_2$}. 
	In this case, we have
	\[i_{K_i}(\tilde{N}_i,U_{i})=0 \quad \text{and} \quad i_{K_i}(\tilde{N}_i,V_{i})=1 \quad (i=1,2). \]
	Hence,
	\[i_{C}(N,U_{1}\times U_{2})=i_{C}(N,V_{1}\times U_{2})=i_{C}(N,U_{1}\times V_{2})=0, \quad i_{C}(N,V_{1}\times V_{2})=1 \]
	and so \eqref{eq_addit} implies that 
	\[i_{C}\left(T,\left(U_1\setminus \overline{V}_1\right)\times \left(U_2\setminus \overline{V}_2\right)\right)=1.\]
	\end{enumerate}
	
	In conclusion, $i_{C}\left(T,\left(U_1\setminus \overline{V}_1\right)\times \left(U_2\setminus \overline{V}_2\right)\right)=\pm 1$ and the existence property of the fixed point index ensures that $T$ has at least one fixed point located in $\left(U_1\setminus \overline{V}_1\right)\times \left(U_2\setminus \overline{V}_2\right)$.
\end{proof}

\subsection{Cartesian product of annular regions and conical shells}

Here, let $(X,\left\|\cdot\right\|)$ and $(Y,\left\|\cdot\right\|)$ be normed linear spaces such that $Y$ is infinite dimensional and $K\subset X$ a cone. 

The following notations will be useful: for given $r,R\in \mathbb{R}_+:=[0,\infty)$, $0<r<R$, we define
\[K_{r,R}:=\{u\in K:r<\left\|u\right\|<R \} \quad \text{ and } \quad \overline{K}_{r,R}:=\{u\in K:r\leq\left\|u\right\|\leq R \}. \]
Moreover, we denote as $A_{r,R}$ the following annular region in the normed space $Y$ 
\[A_{r,R}:=\{v\in Y:r<\left\|v\right\|<R \}, \]
that is, $A_{r,R}=B_{R}\setminus \overline{B}_{r}$ where $B_{\tau}$ stands for the open ball of radius $\tau$ centered at the origin and $\overline{B}_{\tau}$ represents its closure. Furthermore, $\overline{A}_{r,R}:=\overline{B}_{R}\setminus B_{r}$.

As a direct consequence of Theorem \ref{th_starKras}, we establish a result in the line of the \textit{vector version} of Krasnosel'ski\u{\i} fixed point theorem in cones due to Precup \cite{PrecupFPT,PrecupSDC}.

\begin{theorem}\label{th_ancon}
	Take $\alpha_i,\beta_i>0$, with $\alpha_i\neq \beta_i$, $r_i:=\min\{\alpha_i,\beta_i \}$ and $R_i:=\max\{\alpha_i,\beta_i \}$ for $i=1,2$, assume that $T=(T_1,T_2):\overline{K}_{r_1,R_1}\times \overline{A}_{r_2,R_2} \rightarrow K\times Y$ is a compact map and that there exist $h_1\in K\setminus\{0\}$ and $h_2\in Y\setminus\{0\}$ such that for each $i\in\{1,2\}$ the following conditions are satisfied in $\overline{K}_{r_1,R_1}\times \overline{A}_{r_2,R_2}$:
	\begin{enumerate}
		\item[$(a)$] $T_i(u)+\mu\,h_i\neq u_i$ if $\left\|u_i\right\|=\beta_i$ and $\mu\geq 0$;
		\item[$(b)$] $T_i(u)\neq \lambda\, u_i$ if $\left\|u_i\right\|=\alpha_i$ and $\lambda\geq 1$. 
	\end{enumerate}

	Then the fixed point index of $T$ in $K\times Y$ over $K_{r_1,R_1}\times A_{r_2,R_2}$, $i_{K\times Y}(T,K_{r_1,R_1}\times A_{r_2,R_2})$, is well-defined and
	\[i_{K\times Y}(T,K_{r_1,R_1}\times A_{r_2,R_2})=(-1)^k, \]
	where $k\in\{0,1,2\}$ is the number of times that the equality $\beta_i=R_i$ is satisfied, $i=1,2$. 
	
	In particular, $T$ has at least one fixed point $u=(u_1,u_2)\in K\times Y$ such that $r_i<\left\|u_i\right\|<R_i$ for $i=1,2$.
\end{theorem}

\begin{proof}
	In order to apply Theorem \ref{th_starKras}, take the wedges $K_1=K$, $K_2=Y$ and the relatively open sets $U_1=B_{R_1}\cap K$, $V_1=B_{r_1}\cap K$, $U_2=B_{R_2}$ and $V_2=B_{r_2}$. Note that, for each $i\in\{1,2\}$, we have that $U_i$ and $V_i$ are star convex sets since they are convex and contain the zero.
	
	On the other hand, the map $\rho_1: \overline{B}_{R_1}\cap K\rightarrow \overline{K}_{r_1,R_1}$ defined as
	\[\rho_1(v)=\left\{\begin{array}{ll} r_1\displaystyle\frac{v+(r_1-\left\|v\right\|)^2 h_1}{\left\|v+(r_1-\left\|v\right\|)^2 h_1 \right\|}, & \quad \text{if } \left\|v\right\|<r_1, \\[0.2cm] v,  & \quad \text{if } r_1\leq\left\|v\right\|\leq R_1,  \end{array} \right. \]
	is a retraction of $\overline{U}_1=\overline{B}_{R_1}\cap K$ onto $\overline{U}_1\setminus V_1=\overline{K}_{r_1,R_1}$, see \cite[Example 3]{fel} or \cite{JRL}. In addition, notice that $\overline{A}_{r_2,R_2}$ is a retract of $\overline{B}_{R_2}$ since in any infinite dimensional normed space $\partial\,B_{r_2}$ is a retract of $\overline{B}_{r_2}$. 
	
	Therefore, the conclusion follows in a straightforward way from Theorem \ref{th_starKras}.
\end{proof}

\begin{remark}
	Under the assumptions of Theorem \ref{th_ancon}, condition \eqref{eq_multipli} can be seen as
	\[i_{C}(N,\mathcal{O}_{1}\times \mathcal{O}_{2})=i_{K}(\tilde{N}_1,\mathcal{O}_{1})\cdot \deg(I-\tilde{N}_2,\mathcal{O}_{2}), \]
	since it follows from the definition of the fixed point index by means of the Leray-Schauder degree (see \cite{amann,GraDug}) that $i_{Y}(\tilde{N}_2,\mathcal{O}_{2})=\deg(I-\tilde{N}_2,\mathcal{O}_{2})$.
\end{remark}

\begin{remark}
	It is an open problem to decide whether the fixed point index computation $i_{K\times Y}(T,K_{r_1,R_1}\times A_{r_2,R_2})=(-1)^k$ remains valid provided that $T$ is fixed point free on the boundary of the set $K_{r_1,R_1}\times A_{r_2,R_2}$ and hypotheses $(a)$ and $(b)$ in Theorem \ref{th_ancon} are weakened as 
	\begin{enumerate}
		\item[$(\bar{a})$] $T_i(u)+\mu\,h_i\neq u_i$ if $\left\|u_i\right\|=\beta_i$ and $\mu> 0$;
		\item[$(\bar{b})$] $T_i(u)\neq \lambda\, u_i$ if $\left\|u_i\right\|=\alpha_i$ and $\lambda> 1$. 
	\end{enumerate}
	Notice that the previous approach based on the multiplicativity property of the fixed point index does not work since it is not possible to guarantee that the operators $\tilde{N}_1$ and $\tilde{N}_2$ are fixed point free on the boundary of the sets $\mathcal{O}_1$ and $\mathcal{O}_2$, respectively.
\end{remark}

\begin{remark}
	We stress that the main results of this Section can be extended to $n$-dimensional systems of the form
	\[u_i=T_i(u_1,u_2,\dots,u_n), \ i=1,2,\dots, n. \]
	Note that the multiplicativity property is valid for finite products and the homotopies employed in our reasoning in Theorem~\ref{th_starKras} can be adapted in a straightforward manner to operators defined in the product of $n$ wedges. In this situation the index will be again $(-1)^k$, where $k$ represents the number of components $T_i$ for which an expansive behavior holds (that is, $\beta_i=R_i$ in the context of Theorem~\ref{th_ancon}). 
\end{remark}

\section{Applications to elliptic systems}

Consider the following system of quasilinear elliptic equations subject to Dirichlet boundary conditions
\begin{equation}\label{sys_elliptic}
	\left\{\begin{array}{ll} -\Delta u=f(x,u,v), & \ \text{ in } \Omega, \\ -\Delta v=g(x,u,v), & \ \text{ in } \Omega, \\ u=v=0, & \ \text{ on } \partial\,\Omega,  \end{array} \right.
\end{equation}
where $\Omega\subset\mathbb{R}^n$ denotes the unit open ball in $\mathbb{R}^n$, $f:\overline{\Omega}\times \mathbb{R}_{+}\times\mathbb{R}\rightarrow\mathbb{R}_{+}$ and $g:\overline{\Omega}\times \mathbb{R}_{+}\times\mathbb{R}\rightarrow\mathbb{R}$ are continuous functions. 

In the sequel, in order to apply the theory developed in the previous section, we shall work with the normed space $X=Y=\mathcal{C}(\overline{\Omega})$ endowed with the usual norm $\left\|u\right\|_{\infty}=\max_{x\in\overline{\Omega}}\left|u(x) \right|$ (we will simply denote $\left\|\cdot\right\|=\left\|\cdot\right\|_{\infty}$) and the cone of nonnegative continuous functions, i.e., $K:=\{u\in \mathcal{C}(\overline{\Omega}):u\geq 0  \}$.

Now, we consider the following system of Hammerstein integral equations associated to~\eqref{sys_elliptic},
\begin{equation}\label{sys_Ham}
	\left\{\begin{array}{l} 
	u(x)=\displaystyle\int_{\Omega}k(x,y)\,f(y,u(y),v(y))\,dy, \\[0.3cm] 
	v(x)=\displaystyle\int_{\Omega}k(x,y)\,g(y,u(y),v(y))\,dy,  
	\end{array} \right.
\end{equation}
where $k$ is the Green's function corresponding to the problem
\[-\Delta u=h(x) \ \text{ in } \Omega, \quad u=0 \ \text{ on } \partial\,\Omega, \]
and $h$ is a given continuous function. To the system \eqref{sys_Ham} we associate the operator
$$T=(T_1,T_2):K\times Y\rightarrow K\times Y,$$ where
\begin{equation}\label{eq_T}
\begin{array}{l}
	T_1(u,v)(x)=\displaystyle\int_{\Omega}k(x,y)\,f(y,u(y),v(y))\,dy, \\ [0.3cm]
	T_2(u,v)(x)=\displaystyle\int_{\Omega}k(x,y)\,g(y,u(y),v(y))\,dy.
\end{array}
\end{equation}
Note that $T$ is well-defined ($f\geq 0$ implies that $T_1(K\times Y)\subset K$). Moreover, by the continuity of $f$ and $g$, it follows that the operator $T$ is completely continuous, see for example the classical book \cite[Section 7.2]{Kras} or the more recent paper~\cite{Yang-Lan}. 

By a (weak) solution of \eqref{sys_elliptic}, we mean a fixed point of the operator $T$. Hence, in what follows, we will apply Theorem~\ref{th_ancon} to the operator $T$ in order to obtain a solution $(u,v)$ with both components non-trivial. Note that since $g$ is a sign-changing nonlinearity, it is not expected the second component of the solution, $v$, to be a nonnegative function, but it will be localized in an annular region and so it cannot be the identically zero function.

\begin{theorem}\label{th_exis}
	Assume that there exist positive numbers $0<r_1<R_1$, $0<r_2<R_2$ and continuous functions $\overline{f},\underline{f},\overline{g},\underline{g}:\overline{\Omega}\rightarrow\mathbb{R}_+$ such that the following conditions hold:
	\begin{enumerate}[$a)$]
		\item $f(x,u,v)\leq \overline{f}(x)$ on $\overline{\Omega}\times [0,R_1]\times[-R_2,R_2]$ and 
		\[\sup_{x\in\overline{\Omega}}\int_{\Omega}k(x,y)\overline{f}(y)\,dy<R_1;  \]
		\item $\underline{f}(x)\leq f(x,u,v)$ on $\overline{\Omega}\times [0,r_1]\times[-R_2,R_2]$ and 
		\[\sup_{x\in\overline{\Omega}}\int_{\Omega}k(x,y)\underline{f}(y)\,dy>r_1;  \]
		\item $\left|g(x,u,v)\right|\leq \overline{g}(x)$ on $\overline{\Omega}\times [0,R_1]\times[-R_2,R_2]$ and 
		\[\sup_{x\in\overline{\Omega}}\int_{\Omega}k(x,y)\overline{g}(y)\,dy<R_2;  \]
		\item $g(x,u,v)\geq 0$ on $\overline{\Omega}\times [0,R_1]\times [-r_2,r_2]$, $g(x,u,v)\geq \underline{g}(x)$ on $\overline{\Omega}\times [0,R_1]\times[0,r_2]$
		and 
		\[\sup_{x\in\overline{\Omega}}\int_{\Omega}k(x,y)\underline{g}(y)\,dy>r_2.  \]
	\end{enumerate}	

Then the system \eqref{sys_elliptic} has at least one weak solution $(u,v)$ such that $u$ is nonnegative, $r_1<\left\|u\right\|<R_1$ and $r_2<\left\|v\right\|<R_2$.
\end{theorem}

\begin{proof}
	Let us apply Theorem~\ref{th_ancon} to the operator $T=(T_1,T_2):\overline{K}_{r_1,R_1}\times \overline{A}_{r_2,R_2}\rightarrow K\times Y$ defined as in \eqref{eq_T}.	
	
	To do so, let us check first that the following conditions concerning the operator $T_1$ are satisfied in $\overline{K}_{r_1,R_1}\times \overline{A}_{r_2,R_2}$:
	\begin{enumerate}
		\item[$1)$] $T_1(u,v)\neq \lambda\, u$ if $\left\|u\right\|=R_1$ and $\lambda\geq 1$;
		\item[$2)$] $T_1(u,v)+\mu\,{\pmb 1}\neq u$ if $\left\|u\right\|=r_1$ and $\mu\geq0$ (where ${\pmb 1}$ denotes the constant function equal to one).
	\end{enumerate}

	To prove $1)$, we assume by \textit{reductio ad absurdum} that there exist $(u,v)\in K\times Y$ with $\left\|u\right\|=R_1$, $r_2\leq\left\|v\right\|\leq R_2$ and $\lambda\geq 1$ such that for all $x\in\overline{\Omega}$ we have
	\begin{align*}
		\lambda\,u(x)&=\displaystyle\int_{\Omega}k(x,y)\,f(y,u(y),v(y))\,dy \\ &\leq \displaystyle\int_{\Omega}k(x,y)\,\overline{f}(y)\,dy,
	\end{align*}
	and thus, taking the supremum on $\overline{\Omega}$, it follows from condition $a)$ that $\lambda\,R_1=\lambda\,\left\|u\right\|<R_1$, a contradiction. 

	Now, to show that $2)$ holds, assume to the contrary that there exist $(u,v)\in K\times Y$  with $\left\|u\right\|=r_1$, $r_2\leq\left\|v\right\|\leq R_2$ and $\mu\geq 0$ such that $T_1(u,v)+\mu\,{\pmb 1}= u$, that is, for every $x\in \overline{\Omega}$ we have
	\[u(x)=\displaystyle\int_{\Omega}k(x,y)\,f(y,u(y),v(y))\,dy+\mu\,{\pmb 1}. \]
	Then, by hypothesis $b)$ we obtain that for $x\in \overline{\Omega}$,
	\[u(x)\geq \displaystyle\int_{\Omega}k(x,y)\,f(y,u(y),v(y))\,dy\geq \displaystyle\int_{\Omega}k(x,y)\,\underline{f}(y)\,dy. \]
	Hence, passing to the supremum on $\overline{\Omega}$ gives the absurd
	\[r_1=\sup_{x\in\overline{\Omega}}u(x)\geq \sup_{x\in\overline{\Omega}}\displaystyle\int_{\Omega}k(x,y)\,\underline{f}(y)\,dy>r_1. \]

	It remains to prove that the operator $T_2$ satisfies the corresponding conditions in the set $\overline{K}_{r_1,R_1}\times \overline{A}_{r_2,R_2}$, namely,
	\begin{enumerate}
		\item[$3)$] $T_2(u,v)\neq \lambda\, v$ if $\left\|v\right\|=R_2$ and $\lambda\geq 1$;
		\item[$4)$] $T_2(u,v)+\mu\,{\pmb 1}\neq v$ if $\left\|v\right\|=r_2$ and $\mu\geq0$.
	\end{enumerate}
To prove $3)$, we proceed in a similar way as in the proof of $1)$ above;
 note that, in this case, we have to take care of the absolute value of $v$, that is
	\begin{align*}
		|\lambda\,v(x)|&=\left| \displaystyle\int_{\Omega}k(x,y)\,g(y,u(y),v(y))\,dy \right| \\ &\leq \displaystyle\int_{\Omega}k(x,y)\,| g (y,u(y),v(y))|\,dy \leq  \displaystyle\int_{\Omega}k(x,y)\,\overline{g}(y)\,dy< R_2,
	\end{align*}
	which yields a contradiction.

Now let us focus on the point $4)$. Assume that there exist $(u,v)\in \overline{K}_{r_1,R_1}\times \overline{A}_{r_2,R_2}$ with $\left\|v\right\|=r_2$ and $\mu\geq 0$ such that $T_2(u,v)+\mu\,{\pmb 1}= v$. Then we have that $0\leq u(x)\leq R_1$ and $-r_2\leq\left|v(x)\right|\leq r_2$ for all $x\in\overline{\Omega}$ and thus $g(x,u(x),v(x))\geq 0$ for all $x\in\overline{\Omega}$. It follows that 
	\[v(x)=\displaystyle\int_{\Omega}k(x,y)\,g(y,u(y),v(y))\,dy+\mu\,{\pmb 1}\geq \displaystyle\int_{\Omega}k(x,y)\,g(y,u(y),v(y))\,dy, \]
	which implies $v(x)\geq 0$ on $\overline{\Omega}$. By condition $d)$, we deduce that
	\[v(x)\geq \displaystyle\int_{\Omega}k(x,y)\,\underline{g}(y)\,dy \]
	and again, taking the supremum, we get a contradiction.
	
	Therefore, Theorem~\ref{th_ancon} ensures that the operator $T$ has at least one fixed point in ${K}_{r_1,R_1}\times {A}_{r_2,R_2}$.
\end{proof}
We have the following result which is helpful to construct a numerical approximation for the solutions of the system~\eqref{sys_elliptic}. 
\begin{theorem}\label{theoseq}
Under the hypotheses $a)$ and $c)$ of Theorem~\ref{th_exis} it is possible to construct a weak solution of the system~\eqref{sys_elliptic}. 
\end{theorem}	
\begin{proof}	
Take $(u,v)\in  (K\cap \overline{B}_{R_1})\times \overline{B}_{R_2}$ and observe that, due to the hypotheses $a)$ and $c)$,
 for every $x\in \overline{\Omega}$ we have
$$
\begin{array}{l} |T_1 u(x)|=\displaystyle\int_{\Omega}k(x,y)\,f(y,u(y),v(y))\,dy\leq \displaystyle\int_{\Omega}k(x,y)\,\overline{f}(y)\,dy< R_1, \\[0.3cm]  |T_2v(x)|=\Bigl |\displaystyle\int_{\Omega}k(x,y)\,g(y,u(y),v(y))\,dy\Bigr |
\leq \displaystyle\int_{\Omega}k(x,y)\,\overline{g}(y)\,dy<R_2.\end{array}
$$
Therefore $T$ maps $(K\cap \overline{B}_{R_1})\times \overline{B}_{R_2}$ into itself. 

Now, take a  couple of functions $(u_0,v_0)\in (K\cap \overline{B}_{R_1})\times \overline{B}_{R_2}$, and define the sequence  $\{(u_n,v_n)\}$ as the unique solutions (which exist by classical elliptic theory, see for example~\cite{Evans2010}) of the following systems
 \begin{equation}\label{sys_elliptic_seq}
	\left\{\begin{array}{ll} -\Delta u_n=f(x,u_{n-1},v_{n-1}), & \ \text{ in } \Omega, \\ -\Delta v_n=g(x,u_{n-1},v_{n-1}), & \ \text{ in } \Omega, \\ u_n=v_n=0, & \ \text{ on } \partial\,\Omega,  \end{array} \right.
\end{equation}
 for $n=1,2, \ldots$, which are given by
 $$
 (u_n,v_n)=T(u_{n-1},v_{n-1}).
 $$
 
Note that  the sequence $\{(u_n,v_n)\}$  is contained in $(K\cap \overline{B}_{R_1})\times \overline{B}_{R_2}$, because of the previous observation and, furthermore, given the compactness of the operator $T$, $\{(u_n,v_n)\}$  is contained in a compact subset of
$(K\cap \overline{B}_{R_1})\times \overline{B}_{R_2}$. Therefore there exists a subsequence of $\{(u_n,v_n)\}$ (which we denote in the same way, with abuse of notation)  that converges to a couple
$(\overline{u},\overline{v})\in (K\cap \overline{B}_{R_1})\times \overline{B}_{R_2}$.
For this subsequence we have
\begin{equation}\label{sys_approx}
	\left\{\begin{array}{l} u_n(x)=\displaystyle\int_{\Omega}k(x,y)\,f(y,u_{n-1}(y),v_{n-1}(y))\,dy, \\[0.3cm] v_n (x)=\displaystyle\int_{\Omega}k(x,y)\,g(y,u_{n-1}(y),v_{n-1}(y))\,dy.  \end{array} \right.
\end{equation}
By means of the Lebesgue dominated convergence theorem, passing to the limit for $n\to \infty$ in \eqref{sys_approx} we obtain
 $$
 (\overline{u},\overline{v})=T(\overline{u},\overline{v}),
 $$
that is $ (\overline{u},\overline{v})$ is a weak solution of the system~\eqref{sys_elliptic}.
\end{proof}
In the following example we illustrate the applicability of Theorem~\ref{th_exis} and, using the iterative process illustrated in Theorem~\ref{theoseq},
we construct a numerical solution with properties consistent with the theoretical predictions.

\begin{example}\label{esempio1}
	Take the open set $\Omega=\{(x_1,x_2)\in\mathbb{R}^2:x_1^2+x_2^2<1 \}$ and consider the system
	\begin{equation}\label{eq_ex}
		\left\{\begin{array}{ll} -\Delta u=\dfrac{1}{5}(1+x_1^2)e^u(2+\cos v), & \text{ in } \Omega, \\[0.3cm] -\Delta v=\dfrac{3}{4}(1+x_1^2)(1-v^2)(2+\sin u), & \text{ in } \Omega, \\[0.3cm] u=v=0, & \text{ on } \partial\,\Omega. \end{array}\right.
	\end{equation}

	Note that conditions $a)$ -- $d)$ in Theorem~\ref{th_exis} can be verified by choosing $r_1=1/21$, $R_1=1/2$, $r_2=1/6$, $R_2=3/2$ and the constant functions $\overline{f}\equiv 6\sqrt{e}/5$, $\underline{f}\equiv 1/5$, $\overline{g}\equiv 45/8$ and $\underline{g}\equiv 35/24$, as the lower and upper bounds of the nonlinearities
	\[f((x_1,x_2),u,v)=\dfrac{1}{5}(1+x_1^2)e^u(2+\cos v) \ \text{ and } \  g((x_1,x_2),u,v)=\dfrac{3}{4}(1+x_1^2)(1-v^2)(2+\sin u) \] 
	in the corresponding sets.
	To check these computations take into account that
	\[\sup_{x\in\overline{\Omega}}\int_{\Omega}k(x,y)\,{\pmb 1}\,dy=\sup_{x\in\overline{\Omega}}\dfrac{1}{4}(1-x_1^2-x_2^2)=\dfrac{1}{4}, \] 
	as it can be seen by direct calculation.
	
	We now numerically approach the above-written system by using the MATLAB solver for Poisson problems introduced in \cite{ARC}, suitably modified for treating a 	nonlinear system of equations by means of the iterative procedure~\eqref{sys_elliptic_seq}.
    We start with the identically zero initial guess and after fourteen iterations we obtain a numerical solution within a relative tolerance
    of $10^{-10}$ in the infinity norm, this is illustrated in Figure~\ref{fig:test}.
      We remark that the infinity norms of $u$ and $v$ are $0.191$ and $0.406$ respectively, these values are consistent with the estimates obtained with the theoretical results.

	\begin{figure}[h]
		\centering
		\begin{subfigure}{.5\textwidth}
			\centering
			\includegraphics[width=.7\linewidth]{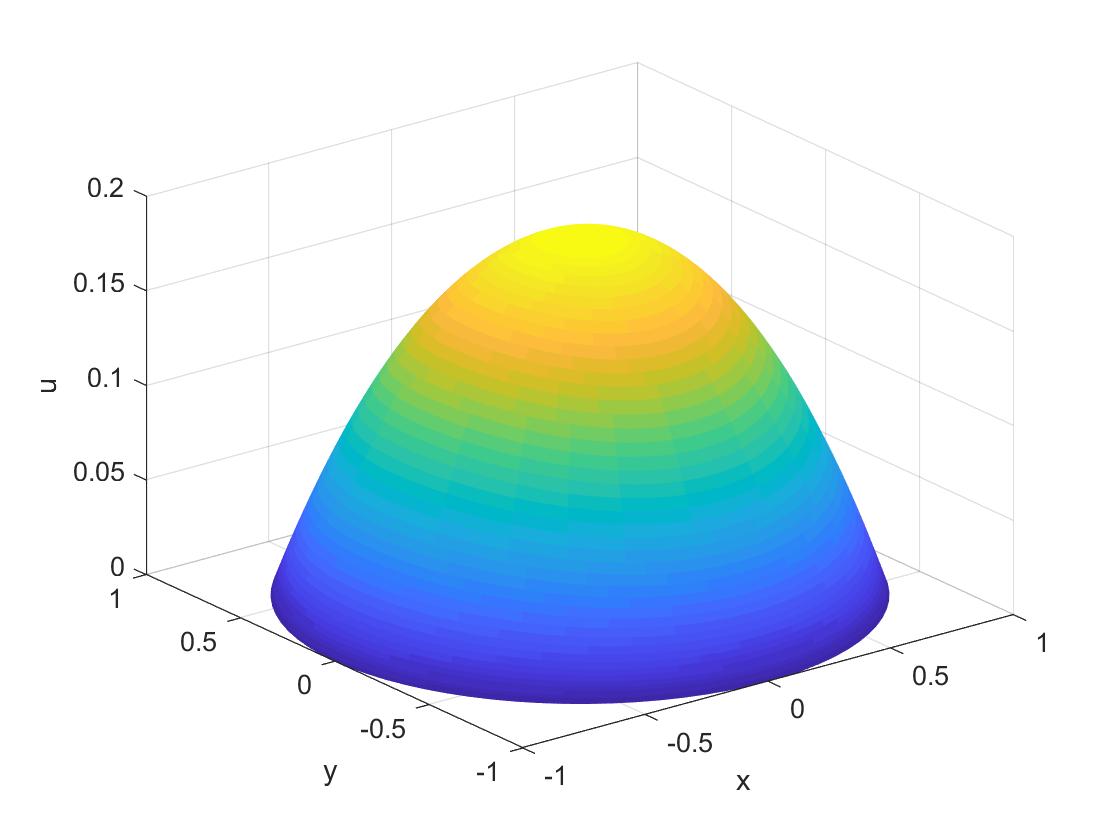}
			\caption{The component $u$}
		\end{subfigure}%
		\begin{subfigure}{.5\textwidth}
			\centering
			\includegraphics[width=.7\linewidth]{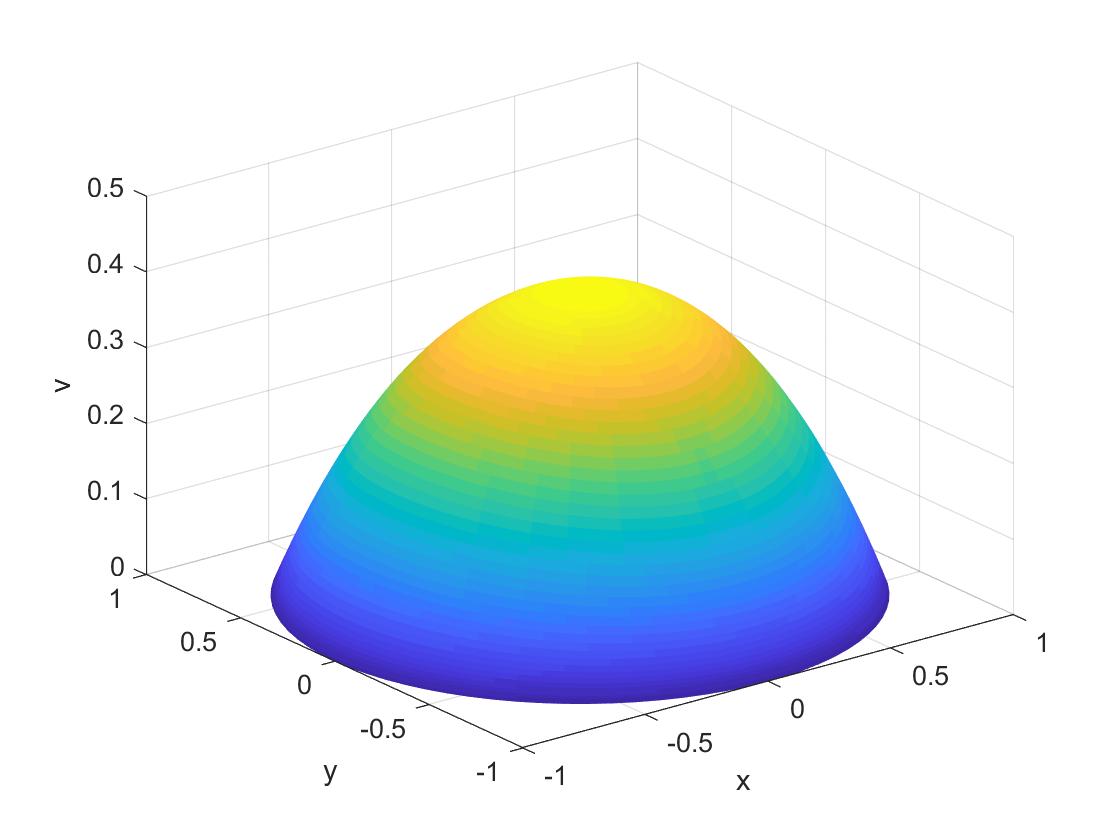}
			\caption{The component $v$}
		\end{subfigure}
		\caption{A numerical solution for Example~\ref{esempio1}}
		\label{fig:test}
	\end{figure}
	
\end{example}
\begin{remark}
Note that in Example~\ref{esempio1} one has  that the operator $T$ maps $(K\cap \overline{B}_{1/2})\times \overline{B}_{3/2}$ into itself, therefore a direct application of the Schauder Theorem would yield a solution of~\eqref{eq_ex}, but with a less precise localization.
\end{remark}

\section{A reaction--diffusion Lotka--Volterra system with source terms for competing species}
As an application of the theory developed so far, we now consider the system given by
\[
\begin{cases}
	\displaystyle \frac{\partial u}{\partial t} = D_1 \Delta u + r_1 u \left(1 - \frac{u}{K_1} \right) - \gamma_1 u v + S_1(x), \\
	\displaystyle \frac{\partial v}{\partial t} = D_2 \Delta v - \mu v - \gamma_2 u v + S_2(x),
\end{cases}
\]
that models the interaction of two competing species. Here, \( u(x,t) \) and \( v(x,t) \) denote the densities of the two species at location $x$ and time $t$, \( D_1, D_2 > 0 \) are their diffusion coefficients, and \( r_1 > 0 \) is the intrinsic growth rate of species \( u \). The parameter \( \mu > 0 \) represents the natural decay rate of species \( v \), while \( K_1 \) is the carrying capacity for species \( u \). The constants \( \gamma_1, \gamma_2 > 0 \) measure the strength of interspecific competition, and the functions \( S_1(x), S_2(x) \geq 0 \) represent spatially dependent source terms in presence of source terms. More precisely, the system describes the following:
\begin{itemize}
	\item The species \( u \) exhibits logistic growth and competes with species \( v \).
	\item The species \( v \) has a negative intrinsic growth rate and cannot persist without external input.
	\item The terms \( S_1(x) \) and \( S_2(x) \) account for spatially distributed environmental support or species introduction.
\end{itemize}

To simplify the analysis, we nondimensionalize the system by introducing the following scaled variables and parameters:
\begin{align*}
	& \tilde{x} = \frac{x}{L}, \quad \tilde{t} = r_1 t, \quad U = \frac{u}{K_1}, \quad V = \frac{v}{V_*}, \quad \text{with } V_* = \frac{r_1}{\gamma_1}, \\
	& \delta_1 = \frac{D_1}{r_1 L^2}, \quad \delta_2 = \frac{D_2}{r_1 L^2},  \quad \alpha = \frac{\mu}{r_1}, \quad \lambda = \frac{\gamma_2K_1 }{\mu}, \\
	& \sigma_1(\tilde{x}) = \frac{S_1(L\tilde{x})}{r_1 K_1}, \quad \sigma_2(\tilde{x}) = \frac{\gamma_1 S_2(L\tilde{x})}{r_1^2},
\end{align*}
where $L$ is a characteristic length.
With abuse of notation, by
dropping the tildes, the dimensionless system becomes:
\[
\begin{cases}
	\displaystyle \frac{\partial U}{\partial t} = \delta_1 \Delta U + U(1 - U) - U V + \sigma_1(x),\vspace{0.1cm} \\
	
	\displaystyle \frac{\partial V}{\partial t} = \delta_2 \Delta V - \alpha V (1 + \lambda U) + \sigma_2(x).
\end{cases}
\]

We focus on the stationary case and examine the existence of steady-state solutions to
\begin{equation}\label{species-model}
\begin{cases}
	- \Delta U = U(1 - U) - U V + \sigma_1(x), \\
	- \Delta V =-\alpha V (1 + \lambda U) + \sigma_2(x),
\end{cases}
\end{equation}
for which an $L^\infty$ a priori estimate follows from Theorem~\ref{th_exis}, where 
 we assume \( \delta_1 = \delta_2 = 1 \) for simplicity. The system~\eqref{species-model} is considered in a two-dimensional unit disk $\Omega$, subject to homogeneous Dirichlet boundary conditions.
Moreover, we take $R_1\le 1$. 
We construct the functions $\overline{f}(x)$, $\underline{f}(x)$, $\overline{g}(x)$, and $\underline{g}(x)$ as follows
\[  \underline{f}(x)=\- r_1 R_2 + \sigma_1(x) \le f(x,U,V) \leq \frac{1}{4} + R_1 R_2 + \sigma_1(x)=\overline{f}(x),\]
\[  \underline{g}(x)=\sigma_2(x) - \alpha r_2 (1+\lambda R_1)\le g(x,U,V)\le |g(x,U,V)| \leq \alpha R_2(1+\lambda R_1) + \sigma_2(x)=\overline{g}(x).\]
By direct calculations, it follows that all the conditions of Theorem~\ref{th_exis} are satisfied if the following inequalities hold: 
\[ 0<\overline{\sigma}_1 < R_1 (4-R_2) - 1/4, \quad \underline{\sigma}_1 > r_1 (4+R_2) > 0, \]
\[ 0<\overline{\sigma}_2 < R_2 (4-\alpha (1+\lambda R_1))>0,\quad \underline{\sigma}_2 > r_2 (4+\alpha (1+\lambda R_1)), \]
\[ \overline{\sigma}_1 > \underline{\sigma}_1\Rightarrow\quad  R_1 (4-R_2) - 1/4 > r_1 (4+R_2) , \]
\[ \overline{\sigma}_2 > \underline{\sigma}_2\Rightarrow \quad  R_2 (4-\alpha (1+\lambda R_1))> r_2 (4+\alpha (1+\lambda R_1)),\]
\[ r_2 \left( 4 + \alpha (1+\lambda R_1) \right) < \underline{\sigma}_2 < \alpha R_2 (1+\lambda R_1), \]
where $\underline{\sigma}_1 := \inf_{x \in \overline{\Omega}} \sigma_1(x)$, $\underline{\sigma}_2 := \inf_{x \in \overline{\Omega}} \sigma_2(x)$,
$\overline{\sigma}_1 := \sup_{x \in \overline{\Omega}} \sigma_1(x)$ and $\overline{\sigma}_2 := \sup_{x \in \overline{\Omega}} \sigma_2(x)$.
In particular, the last right inequality assures that the function $g(x,U,V)$ can assume negative values. 

With the choice
	\[
	\begin{aligned}
		&\sigma_1(x) = 0.3 + 2.4(1 - x_1^2 - x_2^2), \; \sigma_2(x) = 0.09 + 0.5(1 - x_1^2 - x_2^2), \\
		&R_1=1,\; R_2=1,\; r_1=0.05,\; r_2=0.02,\; \alpha=0.1,\; \lambda=0.05, 
	\end{aligned}
\]
where $x = \begin{pmatrix}
	x_1 \\
	x_2
\end{pmatrix}$, all the previous inequalities are satisfied. In particular,
note that $\min {g(x,U,V)}=-0.015<0,$ which is attained for $U=V=1$ and $|x|=1$. The functions $\sigma_1$ and $\sigma_2$  are isotropic, reach their maximum at the centre and decrease with the distance from it. 
The corresponding numerical approximations of the solution are represented in Figure~\ref{fig_ex2}. 
\begin{figure}[h]
	\centering
	\begin{subfigure}{.5\textwidth}
		\centering
		\includegraphics[width=.7\linewidth]{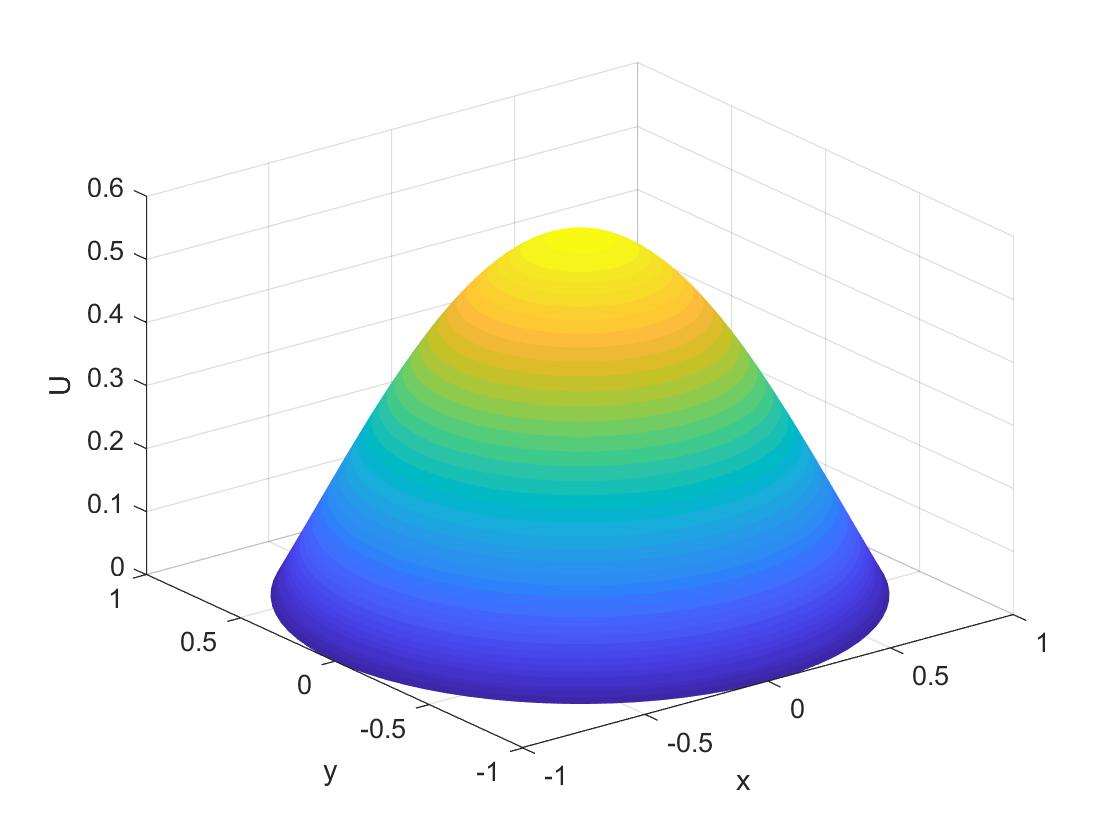}
		\caption{The component $U$}
	\end{subfigure}%
	\begin{subfigure}{.5\textwidth}
		\centering
		\includegraphics[width=.7\linewidth]{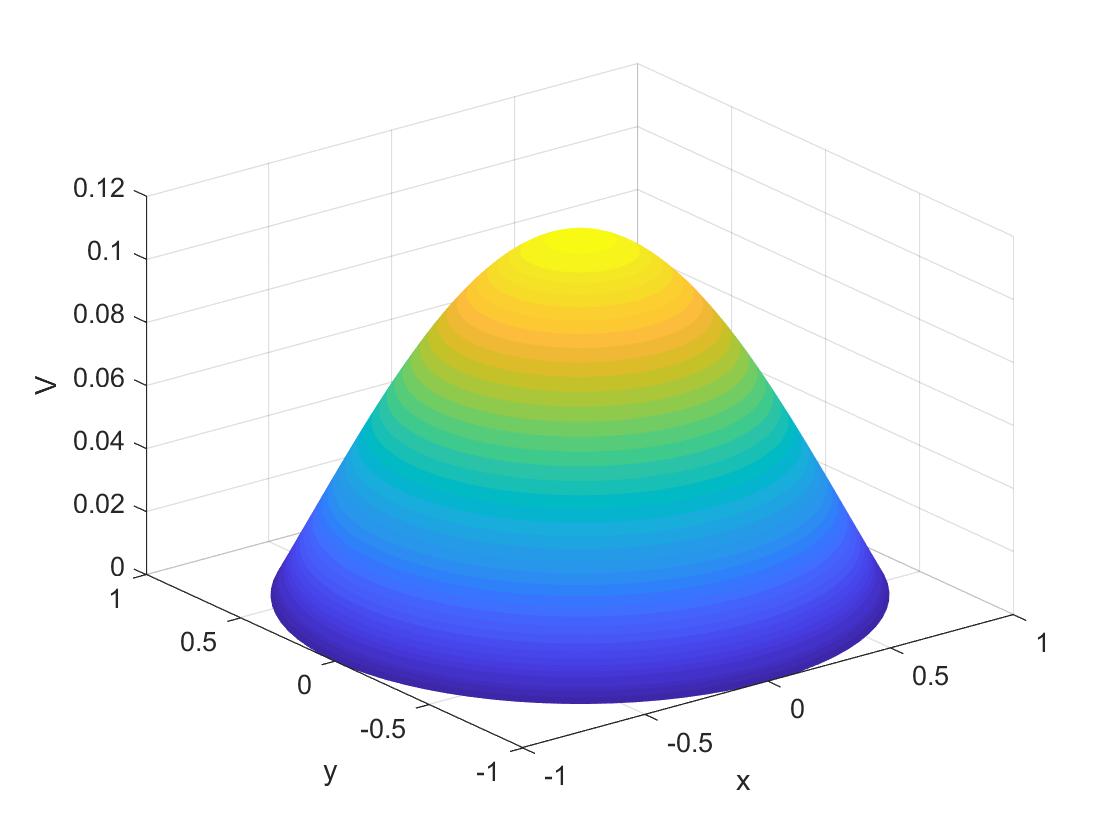}
		\caption{The component $V$}
	\end{subfigure}
	\caption{A numerical solution for~\eqref{species-model}}
	\label{fig_ex2}
\end{figure}
Finally, we emphasize that the infinity norms of $U$ and $V$ are 0.572 and 0.114, respectively, in perfect agreement with the theoretical predictions. This highlights that the results of Theorem~\ref{th_exis} can be effectively used to adjust the parameters to obtain solutions within a prescribed range.

\section*{Acknowledgements}
The authors would like to thank the Referee the careful reading of the manuscript and the constructive comments.
G.~Infante is a member of the Gruppo Nazionale per l'Analisi Matematica, la Probabilit\`a e le loro Applicazioni (GNAMPA) of the Istituto Nazionale di Alta Matematica (INdAM), of the UMI Group TAA  ``Approximation Theory and Applications'' and of  the Research ITalian network on Approximation (RITA). G.~Mascali is a member of the Gruppo Nazionale per la Fisica Matematica (GNFM) of INdAM. G.~Infante and G.~Mascali are partially supported by the project POS-CAL.HUB.RIA.  G. Infante was partly funded by the Research project of MUR - Prin 2022 “Nonlinear differential problems with applications to real phenomena” (Grant Number: 2022ZXZTN2). G. Mascali acknowledges the support from MUR, Project PRIN
“Transport phonema in low dimensional structures: models, simulations and theoretical aspects”
CUP E53D23005900006.
J. Rodr\'iguez--L\'opez has been partially supported by the VIS Program of the University of Calabria, by Ministerio de Ciencia y Tecnología (Spain), AEI and Feder, grant PID2020-113275GB-I00 and by Xunta de Galicia, grant ED431C 2023/12.

\section*{Conflicts of interest}
The authors declare no conflict of interest.

\section*{Contribution statement}
All authors contributed equally to this manuscript.

\end{document}